\documentclass[leqno,12pt]{article}

\usepackage{amsthm}
\usepackage{amsmath}
\usepackage{amssymb}
\usepackage{euscript}

\evensidemargin\oddsidemargin

\begin{document}

\baselineskip=18pt \setcounter{page}{1}

\renewcommand{\theequation}{\thesection.\arabic{equation}}
\newtheorem{theorem}{Theorem}[section]
\newtheorem{lemma}[theorem]{Lemma}
\newtheorem{proposition}[theorem]{Proposition}
\newtheorem{corollary}[theorem]{Corollary}
\newtheorem{remark}[theorem]{Remark}
\newtheorem{fact}[theorem]{Fact}
\newtheorem{problem}[theorem]{Problem}

\newtheorem{thmconnu}{Theorem}
\renewcommand{\thethmconnu}{\Alph{thmconnu}}

\newcommand{\eqnsection}{
\renewcommand{\theequation}{\thesection.\arabic{equation}}
    \makeatletter
    \csname  @addtoreset\endcsname{equation}{section}
    \makeatother}
\eqnsection

\def\r{{\mathbb R}}
\def\e{{\mathbb E}}
\def\p{{\mathbb P}}
\def\bZ{{\mathbb Z}}
\def\S{{\mathbb S}}
\def\t{{\mathbb T}}
\def\s{{\mathcal S}}
\def\bN{{\mathbb N}}
\def\deg{{\rm b}}
\def\P{{\bf P}}
\def\ind {\mbox{\rm 1\hspace {-0.04 in}I}}
\def\E{{\bf E}}
\def\ee{\mathrm{e}}
\def\d{\, \mathrm{d}}
\newcommand{\cF}{\mathcal{F}}
\newcommand{\cG}{\mathcal{G}}
\newcommand{\cU}{\mathcal{U}}

\def\Q{{\bf Q}}
\def\q{{\mathbb Q}}
\def\tS{{\widetilde{S}}}
\def\tQ{{\widetilde{Q}}}
\def\tX{{\widetilde{X}}}
\def\ee{\mathrm{e}}
\newcommand\dd{\, \mathrm{d}}
\renewcommand\d{\mathrm{d}}
\newcommand{\cT}{\mathcal{T}}

\newcommand{\lod}{\cT_\infty}
\newcommand{\roo}{\varrho}
\newcommand{\st}{s}
\newcommand{\lhs}{\textrm{LHS}}
\newcommand{\rhs}{\textrm{RHS}}

\title{The Critical Barrier for the Survival of\\the Branching Random Walk with Absorption}
\author{Bruno Jaffuel}

\maketitle

\begin{abstract}
We study a branching random walk on $\r$ with an absorbing barrier. The position of the barrier depends on the generation. In each generation, only the individuals born below the barrier survive and reproduce.
Given a reproduction law, Biggins et al. \cite{BLSW91} determined whether a linear barrier allows the process to survive. In this paper, we refine their result: in the boundary case in which the speed of the barrier matches the speed of the minimal position of a particle in a given generation, we add a second order term $a n^{1/3}$ to the position of the barrier for the $n^\mathrm{th}$ generation and find an explicit critical value $a_c$ such that the process dies when $a<a_c$ and survives when $a>a_c$. We also obtain the rate of extinction when $a<a_c$ and a lower bound for the population when it survives.
\end{abstract}

\section{Introduction}

We study a discrete-time branching random walk on $\r$.
The population forms a well-known Galton-Watson tree $\cT$, and some extra information is added: to each individual $u\in\cT$ we attach a displacement $\xi_u\in\r$ from the position of her parent. We set the initial ancestor $\roo$ at the origin, hence the individual $u$ has position
\begin{equation*}V(u) = \sum_{\roo<v\le u}\xi_v=\sum_{i=1}^{|u|}\xi_{u_i},\end{equation*}
where $|u|$ is the generation of $u$ and $u_i$ the ancestor of $u$ in generation $i$.
We define an infinite path $u$ through $\cT$ as a sequence of individuals $u=(u_i)_{i\in \bN}$ such that
\begin{equation*}\forall i\in \bN, |u_i|=i \textrm{ and } u_i<u_{i+1}.\end{equation*}
We denote their collection by $\lod$.

Now we explain how the displacements $\xi_u, u\in \cT$ are distributed. A simple choice, with very nice properties would be to take them i.i.d. but actually everything still works in a more general setting. All individuals still reproduce independently and the same way, but we allow correlations in the number and displacements of the children of every single individual. If we write $c(u)$ for the set of children of $u$, our requirement is that the point processes $\{\xi_v,\, u\in c(u)\}$, (with $u$ running over all the potential individuals of the random tree $\cT$) are i.i.d.

We define a barrier as a function $\varphi : \bN \rightarrow \r$. In the branching random walk with absorption, the individuals $u$ such that $V(u)>\varphi(|u|)$, i.e. born above the barrier are removed: they are immediately killed and do not reproduce.

Kesten \cite{Kes78}, Derrida and Simon \cite{DS07},\cite{DS08}, Harris and Harris \cite{HH07} have studied the continuous analog of this process, the branching Brownian motion with absorption. The understanding of what happens in the continuous setting, more convenient to handle from technical point of view, greatly helps us in the discrete one. In particular, we borrow here some ideas from Kesten \cite{Kes78}.

Biggins et al. \cite{BLSW91} introduced the branching random walk with an absorbing barrier in order to answer questions about parallele simulations. Pemantle \cite{Pe09} and Gantert et al. \cite{GHS08} also studied this model.

A natural question that arises is whether the process survives. This obviously depends on the walk as well as on the barrier.
The case of the linear barriers has been solved by Biggins et al. \cite{BLSW91}.

Before stating their result, we need to introduce some notation:

We denote the intensity mesure of this point process by $\mu$, and its Laplace-Stieljes transform by $\Phi$:
\begin{equation*}\Phi(t)=\e \left[\sum_{|u|=1} \ee^{-t \xi_u}\right]=\int_\r \ee^{-tz}\mu(\d z).\end{equation*}
We assume that the expected number of children $\Phi(0)$ is finite and that negative displacements occur, i.e. that $\mu((-\infty,0))>0$.

We also define $\Psi=\log \Phi$, this is a strictly convex function that takes values in $(-\infty,+\infty]$.

We call critical the case where
\begin{equation*}
\Phi(1)=\e\left[\sum_{|u|=1} \ee^{-\xi_u}\right]=1 \textrm{ and } \Phi'(1):=\e\left[\sum_{|u|=1}\xi_u \ee^{-\xi}\right]=0.
\end{equation*}
This can also be written $\Psi(1)=0$ and $\Psi'(1)=0$.

\begin{theorem}[Biggins et al. \cite{BLSW91}]\label{BLSW91}
In the critical case, we have:
\begin{equation*}
\p\left(\exists u \in \lod, \forall i\ge 1, V(u_i)\le i \varepsilon \right) \left\{
\begin{array}{l}
=0 \textrm{ if }\varepsilon\le 0,\\ >0 \textrm{ if }\varepsilon>0.
\end{array}
 \right.
\end{equation*}
\end{theorem}

The aim of this article is to refine this result by replacing the linear barrier $i \mapsto i \varepsilon$ with a more general barrier $i \mapsto \varphi(i)$.

Given a barrier $\varphi$ we do not know in general whether $\p\left(\exists u \in \lod, \forall i\ge 1, V(u_i)\le \varphi(i) \right)=0$ or not.
Theorem \ref{BLSW91} leads us to focus on barriers such that $\frac{\varphi(i)}{i}\rightarrow 0$. Here is the main result we will prove in this paper.

\begin{theorem}\label{result}
We assume:
\begin{equation*}\sigma^2:=\Phi''(1)=\e \left[\sum_{|u|=1} \xi_u^2\ee^{\xi_u}\right]<+\infty.\end{equation*}

Let $a_c=\frac{3}{2}\left(3 \pi^2 \sigma^2\right)^{1/3}$. Then we have:
\[
\p\left(\exists u \in \lod, \forall i\ge 1, V(u_i)\le a i^{1/3} \right) \left\{
\begin{array}{l}
=0 \textrm{ if }a < a_c,\\ >0 \textrm{ if }a>a_c.
\end{array}
 \right.
\]
\end{theorem}
Unfortunately, we have not been able to conclude in the case $a=a_c$, nor give a necessary and sufficient condition on a general barrier for a line of descent to survive below it.

While proving Theorem \ref{result}, we actually obtain stronger results. The two following propositions together imply the theorem.

\begin{proposition}[surviving population]\label{survivors_number}
If $a>a_c$, then the equation $a=b+\frac{3\pi^2\sigma^2}{2b^2}$ has two solutions in $b$, let $b_a$ be the one such that $b_a>\frac{2a_c}{3}$.
For any $\varepsilon>0$, for any $N\in \bN$ large enough, we have with positive probability :
\begin{equation*}
\forall k\ge 1, \# \{u \in \cT_{N^k}: \forall i\le N^k, (a-b_a) i^{1/3} \le V(u_i)\le a i^{1/3} \}\ge \exp\left({N^{k/3} (b_a-\varepsilon)}\right).
\end{equation*}\end{proposition}

\begin{proposition}[rate of extinction]\label{extinction_rate}
If $a<a_c$, then there exists some constant $c>0$ such that
\begin{equation*}\frac{1}{n^{1/3}}\log \p\left(\exists u \in \cT_n, \forall i\le n, V(u_i)\le a i^{1/3} \right)\rightarrow -c.\end{equation*}
\end{proposition}
The constant $c$, which depends on $a$, is determined in Section \ref{section_extinction_rate}.

\subsection{About general barriers}

Let $\varphi : \bN \to \r$ be a barrier. We define $a^+=\limsup_{n\to \infty} \frac{\varphi(n)}{n^{1/3}}$ and $a^-=\liminf_{n\to \infty} \frac{\varphi(n)}{n^{1/3}}$.

We deduce from theorem \ref{result} that there is extinction when $a^+<a_c$ and survival when $a^->a_c$.
Making some modifications to the computations of Section \ref{section_upper_bound}, we can prove the following result :
\begin{theorem}\label{result_bis}
Assume $a^+\ge a_c$. The equation $a^+=b+\frac{3\pi^2\sigma^2}{2 b^2}$ admits a unique solution $b=\frac{2a_c}{3}$ if $a^+=a_c$, and two solutions if $a>a_c$. Let $b_{a^+}\ge \frac{2a_c}{3}$ be the larger solution.

If $a^-<\frac{3\pi^2\sigma^2}{2 b_{a^+}^2}$, then there is extinction.
\end{theorem}

We notice that $\frac{3\pi^2\sigma^2}{2 b_{a^+}^2}\le\frac{3\pi^2\sigma^2}{2 b_{a_c}^2}=\frac{a_c}{3}<a_c$.

When $a^+\ge a_c$, $a^->\frac{3\pi^2\sigma^2}{2 b_{a^+}^2}$, there is not always survival. For example, if $\varphi(n)$ equals $a^+ n^{1/3}$ for $n$ even and $a^- n^{1/3}$ for $n$ odd, then $a^+$ does not matter, it is easy to see that there is extinction if $a^-<a_c$ : staying below this barrier is almost as difficult as for the barrier $n\mapsto a^- n^{1/3}$. The trouble comes from the fact that $\frac{\varphi(n)}{n^{1/3}}$ is too often close to $a^-$.

Actually the condition in Theorem \ref{result_bis} is sharp in the sense that, if we choose some $a^+\ge a_c$ and $a^->\frac{3\pi^2\sigma^2}{2 b_{a^+}^2}$, we can construct a barrier $\varphi$ satisfying $\limsup_{n\to \infty}\frac{\varphi(n)}{n^{1/3}}=a^+$ and $\liminf_{n\to \infty}\frac{\varphi(n)}{n^{1/3}}=a^-$ such that the process survives.
It suffices to take $\varphi(n)=a^- n^{1/3}$ if $n\in\{N^k:k\in \bN\}$ and $\varphi(n)=a^+ n^{1/3}$ else, for some integer $N$ big enough, depending on $a^+$ and $a^-$. The proof of this is essentially identical to the proof of the lower bound contained in Section \ref{section_survivors}.

\subsection{The reduction to the critical case}\label{subsec_red}
It is possible to apply these results to certain non critical branching random walks.
We analyze in which cases this reduction is possible in Appendix \ref{app_reduc}.
Let $t>0$ such that $\Phi(t)<+\infty$ and $\Phi'(t)<+\infty$.
We define a new branching random walk by changing the position of the individual $z$ into $\widetilde{V}(z)=tV(z)+\Psi(t) |z|$ for all $z\in \cT$.

Then, with obvious notation, a straightforward computation gives :
\begin{equation*}\widetilde{\Phi}(1)=\e \left[ \sum_{|u|=1}\ee^{-\widetilde{\xi}_u}\right]=1;\end{equation*}
\begin{equation*}\widetilde{\Phi}'(1)=-\e \left[ \sum_{|u|=1}\widetilde{\xi}_u \ee^{-\widetilde{\xi}_u}\right]
=t \frac{\Phi'(t)}{\Phi(t)} - \Psi(t)=t\Psi'(t)-\Psi(t).\end{equation*}

When we can find $t^*>0$ such that :
\begin{equation}\label{hyp_red}t^*\Psi'(t^*)-\Psi(t^*)=0,\end{equation}
 then the new branching random walk is critical and we can apply Theorem \ref{result}, provided $\widetilde{\sigma}^2:=\e \left[ \sum_{|u|=1}\widetilde{\xi}_u^2 \ee^{-\widetilde{\xi}_u}\right]
 =t^{*2} \Phi''(t^*)-\Psi(t^*)^2$ is finite.
This last condition, which is equivalent to $\Phi''(t^*)<\infty$ will always be fulfilled when $t^*<\zeta$.

The strict convexity of $\Psi$ implies that, when it exists, $t^*$ is unique.

The existence of $t^*$ is discussed in Appendix \ref{app_reduc}.

The rest of the paper is organized as follows:

Section \ref{section_preliminaries} introduces the tools we will use in the proof of our main results.

Section \ref{section_upper_bound} is devoted to the proof of the upper bound in Proposition \ref{extinction_rate}, which contains the first part of Theorem \ref{result}.

In Section \ref{section_survivors}, we prove Proposition \ref{survivors_number} which implies the second part of Theorem \ref{result}.

In Section \ref{section_extinction_rate}, we complete the proof of Proposition \ref{extinction_rate}. We skip many details of technical arguments already exposed in Section \ref{section_survivors} to obtain the lower bound and go back over some results of Section \ref{section_upper_bound} in order to prove that the two bounds agree.

\section{Some preliminaries}\label{section_preliminaries}

\subsection{Many-to-one lemma}

Since $\e \left[\sum_{|u|=1} \ee^{-\xi_u}=1\right]$, we can define the law of a random variable $X$ such that for any measurable nonnegative function $f$, \begin{equation*}\e[f(X)]=\e \left[ \sum_{|u|=1} \ee^{-\xi_u}f(\xi_u)\right].\end{equation*}

Then $\e [X] = \e \left[ \sum_{|u|=1}\xi_u \ee^{-\xi_u}\right]$ so that $X$ is centered by hypothesis.

We denote \begin{equation*}\sigma^2:=\e[X^2]=\Phi''(1)\end{equation*}

\noindent and make the additional assumption that $\sigma^2<+\infty$.

Let $(X_i)_{i\in \bN^*}$ be a i.i.d. sequence of copies of $X$.
Write for any $n\in \bN$, $S_n:=\sum_{0<i\le n} X_i$. $S$ is then a mean-zero random-walk starting from the origin.

We can now state the many-to-one lemma: 

\begin{lemma}[Biggins-Kyprianou \cite{BK97}]\label{afo}
For any $n\ge 1$ and any measurable function $F: \r^n \rightarrow [0,+\infty)$,
\begin{equation*}\e\left[ \sum_{|u|=n} \ee^{-V(u)}F(V(u_i),1\le i \le n)\right]=\e\left[F(S_i),1\le i \le n)\right].\end{equation*}
\end{lemma}

\subsection{Mogul'skii's estimate}

Here we state a particular case (we set the normalizing sequence) of Mogul'skii's estimate:

\begin{theorem}[Mogul'skii \cite{Mog74}]\label{mse}
Let $S$ be a random walk starting from $0$, with zero mean and finite variance $\sigma^2$.
Let $g_1$ and $g_2$ be continuous functions $[0,1]\rightarrow \r$ such that
\begin{equation}\label{cond_mse} g_1(0)\le 0 \le g_2(0) \textrm{ and } \forall t\in [0,1], g_1(t)< g_2(t).\end{equation}
We consider the measurable event
\begin{equation*}E_j=\left\{\forall i\le j, g_1(\frac{i}{j})\le \frac{S_i}{j^{1/3}} \le g_2(\frac{i}{j})\right\}.\end{equation*}
Then
\begin{equation*}\lim_{j\rightarrow \infty} \frac{1}{j^{1/3}} \log \p (E_j)=-\frac{\pi^2 \sigma^2}{2}\int_0^1\frac{\d t}{[g_2(t)-g_1(t)]^2}.\end{equation*}
\end{theorem}

\begin{corollary}\label{mse_ub}
The upper bound in Theorem \ref{mse} is still valid with the second condition in equation (\ref{cond_mse}) replaced by
\begin{equation*} \forall t\in [0,1], g_1(t)\le g_2(t).\end{equation*}
\end{corollary}

\begin{proof}
Let $g_1$, $g_2$ and $E_j$ be as in the statement of the corollary.
For any $\varepsilon>0$, replace $g_2$ by $\widetilde{g}_2:=g_2+\varepsilon$. For any integer $j\ge 1$, $g_1$ and $\widetilde{g}_2$ allow us to define:
\begin{equation*}
\widetilde{E}_j:=\left\{\forall i\le j, g_1(\frac{i}{j})\le \frac{S_i}{j^{1/3}} \le \widetilde{g}_2(\frac{i}{j})\right\}.
\end{equation*}
Now we apply Theorem \ref{mse} to this pair of functions:
\begin{equation*}\limsup_{j\rightarrow \infty} \frac{1}{j^{1/3}} \log \p (\widetilde{E}_j)\le-\frac{\pi^2 \sigma^2}{2}\int_0^1\frac{\d t}{[\widetilde{g}_2(t)-g_1(t)]^2}.\end{equation*}
Using the fact that  $E_j \subset \widetilde{E}_j$, we obtain
\begin{equation*}\limsup_{j\rightarrow \infty} \frac{1}{j^{1/3}} \log \p (E_j)\le-\frac{\pi^2 \sigma^2}{2}\int_0^1\frac{\d t}{[\varepsilon +g_2(t)-g_1(t)]^2}.\end{equation*}
Now we make $\varepsilon$ tend to $0$ and conclude with the monotone convergence theorem.
\end{proof}

\begin{corollary}\label{mse_local}
Under the conditions of Theorem \ref{mse}, we have, for any $\varepsilon>0$,
\begin{equation*}
\lim_{j\rightarrow \infty} \frac{1}{j^{1/3}} \log \p (E_j \cap \{g_2(1)-\varepsilon \le \frac{S_j}{j^{1/3}}\le g_2(1)\} )=-\frac{\pi^2 \sigma^2}{2}\int_0^1\frac{\d t}{[g_2(t)-g_1(t)]^2}.
\end{equation*}
\end{corollary}

\begin{proof}
Let $g_1$, $g_2$ and $E_j$ be as above.
Fix $0<\varepsilon<g_2(1)-g_1(1)$ and $A>0$. From $g_1$ and $g_2$, we define $\widetilde{g}_1$ by
\begin{equation*}
\forall t\in [0,1], \widetilde{g}_1(t)=\max\{g_1(t),g_2(1)-\varepsilon +(t-1)A\}.
\end{equation*}
$\widetilde{g}_1$ and $g_2$ define some pipe $\widetilde{E}_j$ such that for any $j\ge 1 $,
\begin{equation*}\widetilde{E}_j \subset E_j \cap \{g_2(1)-\varepsilon \le \frac{S_j}{j^{1/3}}\le g_2(1)\}.\end{equation*}
We apply Theorem \ref{mse} to this new pair of functions to get
\begin{equation*}
\liminf_{j\rightarrow \infty} \frac{1}{j^{1/3}} \log \p (E_j \cap \{g_2(1)-\varepsilon \le \frac{S_j}{j^{1/3}}\le g_2(1)\} )\ge-\frac{\pi^2 \sigma^2}{2}\int_0^1\frac{\d t}{[g_2(t)-\widetilde{g}_1(t)]^2}.
\end{equation*}
Now we make $A$ go to infinity in the last inequality and use the monotone convergence theorem to obtain the lower bound:
\begin{equation*}
\lim_{j\rightarrow \infty} \frac{1}{j^{1/3}} \log \p (E_j \cap \{g_2(1)-\varepsilon \le \frac{S_j}{j^{1/3}}\le g_2(1)\} )\ge-\frac{\pi^2 \sigma^2}{2}\int_0^1\frac{\d t}{[g_2(t)-g_1(t)]^2}.
\end{equation*}
The upper bound is an obvious consequence of Theorem \ref{mse}, so the corollary is proved.
\end{proof}

\section{Upper bound for the survival probability}\label{section_upper_bound}

\subsection{Splitting the survival probability}
Fix $a>0$.

Obviously,
\begin{equation*}\p\left(\exists u \in \lod, \forall i, V(u_i)\le a i^{1/3} \right)=\lim_{n\rightarrow \infty} \p \left(\exists u \in \cT_n, \forall i\le n, V(u_i)\le a i^{1/3} \right).\end{equation*}

Now on, $n\ge1$ is fixed.

We now set a second barrier $i\mapsto a i^{1/3} - b_{i,n}$  (with $ b_{i,n}>0$ for $1\le i \le n$ yet to be determined) below the first one $i\mapsto a i^{1/3}$: if a particle crosses it, then its descendance will be likely to stay below the first one until generation $n$.

Let $H(u)$ be the first time a particle $u\in \cT_n$ crosses this second barrier ($H(u)=\infty$ if the particle stays between the barriers until time $n$). We split the sum accordingly:
\begin{equation}\p \left(\exists u \in \cT_n, \forall i\le n, V(u_i)\le a i^{1/3} \right)\le R_\infty+\sum_{j=1}^n R_j , \label{splitub}
\end{equation}
where
\begin{equation*}R_j=\p  \left(\exists u \in \cT_n,H(u)=j,\forall i\le n, V(u_i)\le a i^{1/3}\right) \textrm{ for }j=1,\dots,n,\infty.\end{equation*}

By Chebyshev's inequality and then Lemma \ref{afo}, we get
\begin{eqnarray}
R_\infty & \le & \e \left[\sum_{u \in \cT_n} \ind_{\left\{\forall i\le n, a i^{1/3}-b_{i,n}\le V(u_i)\le a i^{1/3}\right\}} \right]\nonumber\\
& \le & \e  \left[\ee^{S_n} \ind_{\left\{\forall i\le n, a i^{1/3}-b_{i,n}\le S_i\le a i^{1/3}\right\}} \right]\nonumber\\
& \le & \ee^{a n^{1/3}}\p  \left(\forall i\le n, a i^{1/3}-b_{i,n}\le S_i\le a i^{1/3} \right)\label{r_inf_un}.
\end{eqnarray}

For $1\le j\le n$,

\begin{eqnarray}
R_j & \le & \e  \left[\sum_{v \in \cT_j} \ind_{\left\{\forall i < j, a i^{1/3}-b_{i,n}\le V(v_i)\le a i^{1/3}, V(v)<a j^{1/3}-b_{j,n}\right\}} \right]\nonumber\\
&\le & \e  \left[\ee^{S_j} \ind_{\left\{\forall i < j, a i^{1/3}-b_{i,n}\le S_i\le a i^{1/3}, V(S_j)<a j^{1/3}-b_{j,n}\right\}} \right]\nonumber\\
& \le & \ee^{a j^{1/3}-b_{j,n}}\p  \left(\forall i < j, a i^{1/3}-b_{i,n}\le S_i\le a i^{1/3} \right)\label{r_j_un}.
\end{eqnarray}

\subsection{Asymptotics for $R_\infty$}
Fix $\varepsilon>0$.

In order to apply Corollary \ref{mse_ub}, we set $b_{i,n}:=n^{1/3} g(\frac{i}{n})$ for some continuous function $g: [0,1]\mapsto [0,+\infty)$. Then we have
\begin{equation}\label{r_inf_deux} \exists N\ge1, \forall n\ge N, \frac{1}{n^{1/3}} \log \p  \left(\forall i\le n, a i^{1/3}-b_{i,n}\le S_i\le a i^{1/3} \right)\le -\frac{\pi^2 \sigma^2}{2}\int_0^1\frac{\d t}{g(t)^2}+\varepsilon.
\end{equation}
Putting together equations (\ref{r_inf_un}) and (\ref{r_inf_deux}), we get
\begin{equation*}R_\infty \le \exp\left(n^{1/3} [a-\frac{\pi^2 \sigma^2}{2}\int_0^1\frac{\d t}{g(t)^2}+\varepsilon]\right).\end{equation*}
Since $\varepsilon$ is arbitrary small, we get
\begin{equation}\label{estim_r_inf}
\limsup_{n\rightarrow \infty}\frac{\log R_\infty}{n^{1/3}} \le -\st_1,
\end{equation}
where
\begin{equation}\label{rate_r_inf}
\st_1:=-a+\frac{\pi^2 \sigma^2}{2}\int_0^1\frac{\d t}{g(t)^2}.
\end{equation}

\subsection{Asymptotics for $R_j$}

Now suppose $j=j(n):= \alpha n +1$ for some rational $\alpha\in[0,1)$ and for values of $n$ such that $\alpha n$ is an integer.
Define for any $t\in[0,1]$, $g_\alpha(t):=\alpha^{-1/3}g(\alpha t)$.
We can now apply Corollary \ref{mse_ub} so that for $j(n)$ large enough, we have
\begin{eqnarray}
\frac{1}{(j-1)^{1/3}} \log \p  \left(\forall i<j, a i^{1/3}-(j-1)^{1/3}g_\alpha(\frac{i}{j-1})\le S_i\le a i^{1/3} \right) & \le &-\frac{\pi^2 \sigma^2}{2}\int_0^1\frac{\d t}{g_\alpha (t)^2}+\varepsilon, \nonumber\\
\frac{1}{(\alpha n)^{1/3}} \log \p  \left(\forall i<j, a i^{1/3}-n^{1/3}g(\frac{i}{n})\le S_i\le a i^{1/3} \right) & \le &-\frac{\pi^2 \sigma^2}{2\alpha^{1/3}}\int_0^\alpha\frac{\d u}{g(u)^2}+\varepsilon.\label{r_j_deux}
\end{eqnarray}

Putting together equations (\ref{r_j_un}) and (\ref{r_j_deux}), we get
\begin{equation}\label{r_alpha_un}
R_{\alpha,n}:=R_j(n) \le \exp\left(n^{1/3} [a \alpha^{1/3}-g(\alpha)-\frac{\pi^2 \sigma^2}{2}\int_0^\alpha\frac{\d t}{g(t)^2}+\varepsilon]\right).
\end{equation}

Unfortunately, these inequalities hold only for $n$ greater than some $n_0$ which depends on $\alpha$ (and such that $\alpha n$ is an integer).
Thus we can only apply Theorem \ref{mse} to a finite number of values of $\alpha$ at the same time.

Since $\alpha \mapsto g(\alpha)$ and $\alpha \mapsto \int_0^\alpha\frac{\d t}{g(t)^2}$ are uniformly continuous on the compact $[0,1]$, we can choose $N\ge 1$ such that $\forall \alpha_1<\alpha_2\in[0,1], \left(\alpha_2-\alpha_1\le \frac{1}{N}\Rightarrow |\frac{\pi^2 \sigma^2}{2}\int_{\alpha_1}^{\alpha_2}\frac{\d t}{g(t)^2}|<\varepsilon\right)$.
We apply $N-1$ times Theorem \ref{mse} in order to obtain equation (\ref{r_alpha_un}) for simultaneously $\alpha=\frac{1}{N},\frac{2}{N},\dots,\frac{N-1}{N}$ for $n\ge n_0$ for some $n_0\ge 1$ and with $n\in N \bN$.
Observe that equation (\ref{r_alpha_un}) trivially holds for $\alpha=0$.

For any $1\le j \le n $, we consider $\alpha:=\frac{j-1}{n}$ and $0\le k < N$ such that $\frac{k}{N}\le \alpha <\frac{k+1}{N}$. Using (\ref{r_alpha_un}) for $\widetilde{\alpha}=\frac{k}{N}$ and the fact that $m\mapsto \p \left(\forall i<m, a i^{1/3}-n^{1/3}g(\frac{i}{n})\le S_i\le a i^{1/3} \right)$ is nonincreasing, we have
\begin{eqnarray}
R_{\alpha,n}&\le& \p \left(\forall i\le\frac{kn}{N}, a i^{1/3}-n^{1/3}g(\frac{i}{n})\le S_i\le a i^{1/3} \right)\exp\left(n^{1/3} [a \alpha^{1/3}-g(\alpha)]\right)\nonumber\\
&\le& \exp\left(n^{1/3} [a \alpha^{1/3}-g(\alpha)-\frac{\pi^2 \sigma^2}{2}\int_0^{\widetilde{\alpha}}\frac{\d t}{g(t)^2}+\varepsilon]\right)\nonumber\\
&\le& \exp\left(n^{1/3} [a \alpha^{1/3}-g(\alpha)-\frac{\pi^2 \sigma^2}{2}\int_0^\alpha\frac{\d t}{g(t)^2}+2\varepsilon]\right).\label{r_alpha_deux}
\end{eqnarray}

As a consequence,
\begin{equation}\label{estim_sum_r_j}
\limsup_{n\rightarrow \infty,n\in N \bN}\frac{1}{n^{1/3}}\log \sum_{j=1}^n R_j(n) \le -\st_2+2\varepsilon.
\end{equation}
where
\begin{equation}\label{rate_sum_r_j}
\st_2:=\min_{0\le\alpha\le1} \left\{-a \alpha^{1/3}+g(\alpha)+\frac{\pi^2 \sigma^2}{2}\int_0^\alpha\frac{\d t}{g(t)^2}\right\}.
\end{equation}

\begin{remark}This is enough to prove the extinction when $a<a_c$ in Theorem \ref{result} but does not rigorously leads to the upper bound of Proposition \ref{extinction_rate} because of the restriction $n\in N \bN$ (with $N$ depending on $\varepsilon$). This will be fixed in Section \ref{suppress_restriction}.
\end{remark}

Combining (\ref{estim_sum_r_j}) with (\ref{estim_r_inf}) and (\ref{splitub}), we obtain
\begin{equation*}
\limsup_{n\rightarrow \infty,n\in N \bN}\frac{1}{n^{1/3}}\log \p \left(\exists u \in \cT_n, \forall i\le n, V(u_i)\le a i^{1/3} \right) \le -\st+2\varepsilon,
\end{equation*}
where $\st:=\min(\st_1,\st_2)$.

\subsection{Choice of $g$ for the upper bound}

Set $a>0$ and $\st>0$.
We are looking for a function $g$ such that $\st>0$ when $a<a_c$. Taking $\varepsilon$ small enough and $N$ large enough, the existence of such a function implies extinction and ends the proof the first part of Theorem \ref{result}.

We add the constraint $g(1)=0$. Taking $\alpha=1$, we see from (\ref{rate_sum_r_j}) and (\ref{rate_r_inf}) that this implies $\st_2\le\st_1$ and, as a result, $\st=\st_2$.

We choose $g$ in such a way that the quantity $-a \alpha^{1/3}+g(\alpha)+\frac{\pi^2 \sigma^2}{2}\int_0^\alpha\frac{\d t}{g(t)^2}$ which appears in (\ref{rate_sum_r_j}) does not depend on $\alpha$.
Hence $g$ is defined as the solution of the equation:
\begin{equation}\label{equ_int}
\forall t \in [0,1], -a t^{1/3}+f(t)+\frac{\pi^2 \sigma^2}{2}\int_0^t\frac{\d u}{f(u)^2}=\st,
\end{equation}
where $\st$ is some positive constant, the value of which is to be set later.

Equivalently, this can be written $f(0)=\st$ and $\forall t \in (0,1)$,
\begin{equation}\label{equ_diff}
f'(t)=\frac{a}{3} t^{-2/3}-\frac{\pi^2 \sigma^2}{2f(t)^2}.
\end{equation}
By the Picard-Lindel\"of theorem, such an ordinary differential equation admits a unique maximal solution $f$ defined on an interval $[0,t_{max})$ with $t_{max}\in (0,+\infty]$. And if $t_{max}<+\infty$, then $f$ has limit $0$ or $+\infty$ when $t$ goes to $t_{max}$.
\begin{remark}
The fact that $f'(0)$ does not exist here is not troublesome at all since the proof of the theorem, using Picard iterates, actually relies on equation (\ref{equ_int}).
\end{remark}

In order to prove that there exists $\st$ such that $t_{max}=1$ and $lim_{t\rightarrow 1}f(t)=1$, we get a close look to the differential equation.

First we state three simple results specific to this differential equation.

\begin{proposition}[invariance property]\label{inv_prop}
Let $\lambda>0$ and $f$ a continuous function $[0,t_0)\mapsto (0,+\infty)$. Define $f_\lambda: (0,\lambda^{-1}t_0)\mapsto(0,+\infty)$ by
\begin{equation*}
f_\lambda(t)=\lambda^{-1/3}f(\lambda t).
\end{equation*}
Then $f$ satisfies equation (\ref{equ_diff}) on $(0,t_0)$ if and only if $f_\lambda$ does on $(0,\lambda^{-1} t_0)$.
\end{proposition}

\begin{proof}
Assume that $f$ satisfies equation (\ref{equ_diff}) for any $0<t<t_0$. Then for any $0<t<\lambda^{-1}t_0$,
\begin{eqnarray*}
f_\lambda'(t)&=&\lambda^{2/3}f'(\lambda t)\\
&=&\lambda^{2/3}\left(\frac{a}{3} (\lambda t)^{-2/3} - \frac{\pi^2\sigma^2}{2f(\lambda t)^2}\right)\\
&=&\frac{a}{3} t^{-2/3} - \frac{\pi^2\sigma^2}{2f_\lambda(t)^2}.
\end{eqnarray*}
This means that $f_\lambda$ also satisfies equation (\ref{equ_diff}) for any $0<t<\lambda^{-1}t_0$.

Conversely, assume that $f_\lambda$ satisfies equation (\ref{equ_diff}) on $(0,\lambda^{-1} t_0)$. We notice that if $\lambda'>0$, then $(f\lambda)_{\lambda'}=f_{\lambda \lambda'}$. We take $\lambda'=\lambda^{-1}$. Hence $(f\lambda)_{\lambda'}=f$ also satifies equation (\ref{equ_diff}) for any $0<t<(\lambda \lambda')^{-1}t_0=t_0$.
\end{proof}

\begin{proposition}\label{crois_en_a}
Set $0<a_1<a_2$ and $\st>0$. Let $f_1$ and $f_2$ be functions $[0,t_{max})\mapsto (0,+\infty)$ such that
\begin{equation*}
\forall 0<t<t_{max}, \forall i\in\{1,2\},-a t^{1/3}+f_i(t)+\frac{\pi^2 \sigma^2}{2}\int_0^t\frac{\d u}{f_i(u)^2}=\st.
\end{equation*}
Then, for all $0\le t<t_{max}$, $f_1(t)\le f_2(t)$.
\end{proposition}

\begin{proof}
It suffices to prove that, if $0\le t_{start}$, $0<a_1<a_2$ and $0<x_1\le x_2$, then there exist $t_{next}>t_{start}$ such that there are functions $f_1$ and $f_2:[t_{start},t_{next})\mapsto (0,+\infty)$ such that
\begin{equation*}
\forall t_{start}\le t<t_{next}, \forall i\in\{1,2\},-a_i (t^{1/3}-t_{start}^{1/3})+f_i(t)+\frac{\pi^2 \sigma^2}{2}\int_{t_{start}}^t\frac{\d u}{f_i(u)^2}=x_i;
\end{equation*}
then, for any $t_{start}\le t<t_{next}$, $f_1(t)\le f_2(t)$.

We choose $t_{next}$ such that the Picard interates $f_i^n$ defined, for $i\in\{1,2\}$, by :
\begin{equation*}
\forall t_{start}\le t<t_{next},f_i^0(t)=x_i;
\end{equation*}
\begin{equation*}
\forall n\in\bN,\forall t_{start}\le t<t_{next}, f_i^{n+1}(t)=f_i^n(t_{start})+a_i (t^{1/3}-t_{start}^{1/3})-\frac{\pi^2 \sigma^2}{2}\int_{t_{start}}^t\frac{\d u}{f_i^n(u)^2},
\end{equation*}
exist and converge on $[t_{start},t_{next})$. The limits $f_i$ are solutions of the integral equations for $i\in\{1,2\}$.

It is easy to prove by induction on $n$ that
\begin{equation*}
\forall n\in \bN, \forall t_{start}\le t<t_{next}, f_1^n(t)\le f_2^n(t).
\end{equation*}
Letting $n$ tend to infinity gives us the desired conclusion.
\end{proof}

\begin{proposition}\label{disjonction}
Let $f$ as above.
Then we are in one of the following cases:

(A) $t_{max}=+\infty$ and $f(t)\rightarrow +\infty$ as $t\rightarrow +\infty$;

(B) $t_{max}<+\infty$ and $f(t)\rightarrow 0$ as $t\rightarrow t_{max}$.
\end{proposition}

\begin{proof}
First notice that for any $0<t<t_{max}$, $f(t)\le \st + a t^{1/3}$.
A consequence of this inequality is that if $t_{max}<+\infty$, then the limit of $f$ when $t$ goes to $t_{max}$ can only be $0$.

Now, suppose that $t_{max}=+\infty$ but that $f$ does not go to infinity. Then there are $M>0$ and a sequence $(t_n)_n\ge1$ with $\lim_n t_n=+\infty$ such that for any $n\ge 1$, $f(t_n)\le M$.
We can choose $n$ such that $\frac{a}{3} t_n^{-2/3} - \frac{\pi^2\sigma^2}{2M^2}<0$.

Then it is easy to see that $f$ decreases after $t_n$.
Indeed, for $t\ge t_n$,
\begin{equation*}
f'(t)=\frac{a}{3} t^{-2/3} - \frac{\pi^2\sigma^2}{2f(t)^2}\le \frac{a}{3} t_n^{-2/3} - \frac{\pi^2\sigma^2}{2M^2}<0.
\end{equation*}
To be rigorous, we must consider $t_*:=\inf\left\{t\ge t_n, f'(t)>\frac{a}{3} t_n^{-2/3} - \frac{\pi^2\sigma^2}{2M^2}\right\}$ and notice that if we assume $t_*<+\infty$, then $f$ decreases in a neighborhood of $t_*$ and the inequality still holds on this neighborhood, which contradicts the definition of $t_*$.

We have proved that $f'(t)$ is less than a negative constant for $t\ge t_n$, which implies that $f$ reaches zero in finite time.
\end{proof}

Assume we are in the second case of Proposition \ref{disjonction}.
We set $\lambda:=t_{max}^{-1}$ and define the function $f_\lambda$ as in Proposition \ref{inv_prop} (with $t_0=t_{max}$).
We choose $g=f_\lambda$ and set $g(1)=0$ so that $g$ is continuous over $[0,1]$ and satifies (\ref{equ_diff}) for all $t\in(0,1)$.

\begin{remark}
A consequence of Proposition \ref{inv_prop} is that the choice of the value $\st$ of $f(0)$ does not matter at all. If we replace $\st>0$ with another $\widetilde{\st}>0$, we then replace $\lambda$ with $\widetilde{\lambda}=\lambda \left(\frac{\widetilde{\st}}{\st}\right)^3$ and finally get the same $g$.
\end{remark}

So we only have to prove that, when $a<a_c$, we are in case (B) of Proposition \ref{disjonction}, and we will deduce the upper bound in Theorem \ref{result}. This is contained in the following:
\begin{proposition}\label{study_equ_diff}
Let $f$ be the solution of equation (\ref{equ_diff}) with initial condition $f(0)=1$.
\renewcommand{\theenumi}{\roman{enumi}}
\renewcommand{\labelenumi}{(\theenumi)}
\begin{enumerate}
\item If $a> a_c$, then $t_{max}=+\infty$ and $f(t)\sim bt^{1/3}$ as $t\rightarrow +\infty$ with $b$ defined by $b>\frac{2 a_c}{3}$ and $a=b+\frac{3\pi^2\sigma^2}{2 b^2}$.
\item If $a= a_c$, then $t_{max}=+\infty$ and $f(t)\sim \frac{2 a_c}{3}t^{1/3}$ as $t\rightarrow +\infty$.
\item If $a< a_c$, then $t_{max}<+\infty$ and $f(t)\rightarrow 0$ as $t\rightarrow t_{max}$.
\end{enumerate}
\end{proposition}

In the proof of the proposition, we will need the following lemma:
\begin{lemma}\label{ameliore_b}
Assume that $f$ is a solution on $[0, +\infty)$ of the differential equation and that:
\begin{equation*}
\limsup_{t\to +\infty} \frac{f(t)}{t^{1/3}}\le b.
\end{equation*}
Then we have
\begin{equation*}
\limsup_{t\to +\infty} \frac{f(t)}{t^{1/3}}\le b':=a-\frac{3\pi^2\sigma^2}{2b^2}.
\end{equation*}
\end{lemma}
\begin{proof}[Proof of Lemma \ref{ameliore_b}]
Let $\varepsilon>0$. By hypothesis, for any $t$ greater than some $t_0$, we have $f(t)\le  (b+\varepsilon) t^{1/3}$.
For some real constants $c_0$ and $c'_0$ and any $t\ge t_0$, we have, by equation (\ref{equ_int}):
\begin{equation*}
f(t)\ge c_0 + a t^{1/3} - \frac{\pi^2\sigma^2}{2 (b+\varepsilon)^2}\int_{t_0}^t \frac{\d u}{u^{2/3}}=c'_0 +\left(a-\frac{3\pi^2\sigma^2}{2(b+\varepsilon)^2}\right)t^{1/3}.\end{equation*}
Hence
\begin{equation*}\limsup_n \frac{f(t)}{t^{1/3}}\le \left(a-\frac{3\pi^2\sigma^2}{2(b+\varepsilon)^2}\right).\end{equation*}
Letting $\varepsilon$ tend to $0$ ends the proof of the lemma.\end{proof}
Iterating Lemma \ref{ameliore_b}, we obtain:
\begin{lemma}\label{ameliore_b_itere}
Assume that $f$ is a solution on $[0, +\infty)$ of the differential equation and that for all $\limsup_{t\to +\infty} \frac{f(t)}{t^{1/3}}\le b_0$.
We define the sequence $(b_n)_{n\in\bN}$ recursively by $b_{n+1}:=a-\frac{3\pi^2\sigma^2}{2b_n^2}$.
Then
\begin{equation*}
\forall n\ge1, \limsup_{t\to +\infty} \frac{f(t)}{t^{1/3}}\le b_n.
\end{equation*}
\end{lemma}

\begin{proof}[Proof of Proposition \ref{study_equ_diff}]

(i) Assume $a\ge a_c$ and let $b$ such that $a=b+\frac{3\pi^2\sigma^2}{2 b^2}$.
Define, for $0\le t\le t_{max}$, $f_0(t):=b t^{1/3}$.
Then $f_0$ satisfies equation (\ref{equ_diff}) as $f$ does, with initial condition $f_0(0)=0<f(0)=\st$. Hence
\begin{equation*}
\forall 0\le t\le t_{max}, f(t)\ge f_0(t).
\end{equation*}
This implies $t_{max}=+\infty$.
Now let $h=f-f_0$. Then, by equation (\ref{equ_int}), we have, for $t\ge 0$,
\begin{eqnarray*}
h(t)&=&\st + (a-b)t^{1/3}-\int_0^t\frac{\pi^2\sigma^2 \dd u}{2f(u)^2}\\
&=&\st + (a-b-\frac{3\pi^2\sigma^2}{2b^2})t^{1/3}+\int_0^t\frac{\pi^2\sigma^2\dd u}{2}\left(\frac{1}{f_0(u)^2}-\frac{1}{f(u)^2}\right).
\end{eqnarray*}
Since $a=b+\frac{3\pi^2\sigma^2}{2 b^2}$,
\begin{equation*}
h(t)=\int_0^t\frac{\pi^2\sigma^2}{2}\left(\frac{\d u}{f_0(u)^2}-\frac{1}{f(u)^2}\right)\le \int_0^t\frac{\pi^2\sigma^2}{2}\frac{2h(u)\dd u}{f_0(u)^3}.
\end{equation*}
We apply Gronwall's lemma and obtain, for any $0<t_0<t$,
\begin{equation}
h(t)\le h(t_0) \exp\left(\int_{t_0}^t\frac{\pi^2\sigma^2 \dd u}{b^3 u}\right)=h(t_0) (\frac{t}{t_0})^\frac{\pi^2\sigma^2}{b^3}.\label{glemma}
\end{equation}
Notice that $\frac{\pi^2\sigma^2}{b^3}=\frac{1}{3}\frac{a_c^3}{b^3}$.
Then if $a>a_c$ and $b>\frac{2a_c}{3}$, the exponent in the right-hand side of (\ref{glemma}) will be less than $\frac{1}{3}$.

 (ii) Now assume $a= a_c$ and $b=\frac{2a_c}{3}$. This is the same as when $a>a_c$, except that the exponent in the right-hand side of (\ref{glemma}) is exactly $\frac{1}{3}$, which means that for some constant $b_0>\frac{2a_c}{3}$,
\begin{equation*}
\forall t\ge t_0,\; f_0(t)\le f(t)\le b_0 t^{1/3}.
\end{equation*}
Now apply Lemma \ref{ameliore_b_itere}.
The result follows from that $\lim_n b_n=\frac{2a_c}{3}$.

(iii) Assume $a< a_c$ and $t_{max}=+\infty$. Then, by (ii) and Proposition \ref{crois_en_a}, we have that any $b_0>\frac{2a_c}{3}$, for $t$ large enough,
\begin{equation*}
\limsup_{t\to +\infty} \frac{f(t)}{t^{1/3}}\le b_0.
\end{equation*}
Now we apply Lemma \ref{ameliore_b_itere}.
If $b_0$ is close enough to $\frac{2a_c}{3}$, we will have $b_1<\frac{2a_c}{3}$ and $b_n\rightarrow -\infty$ as $n$ goes to infinity, which is absurd.
We conclude that the hypothesis $t_{max}=+\infty$ is false, which proves the proposition.
\end{proof}

\section{Lower bound for the survival probability}\label{section_survivors}

\subsection{Strategy of the estimate}

The basic idea is to consider only the population between two barriers (below $i\mapsto a i^{1/3}$ but above $i\mapsto (a-b) i^{1/3}$), estimate the first two moments of the number of individuals in generation $n$ and then to use the Paley-Zygmund inequality to get the lower bound.

Unfortunately, Mogul'skii's estimate causes the apparition of a factor $\ee^{o(n^{1/3})}$ in the estimates of the moments of the surviving population at generation $n$, so we will not be able to prove directly that the population survives with positive probability.

Here is how to overcome this difficulty:

Set $\lambda>0$ such that $\ee^{\lambda}\in \bN$ and $(v_k)_{k\ge 1}$ a sequence of positive integers.
We consider the population surviving below the barrier $i\mapsto a i^{1/3}$: any individual that would be born above this barrier is removed and consequently does not reproduce.
For any $k \in \bN$, we pick a single individual $z$ at position $V(z)$ in generation $\ee^{\lambda k}$ and consider the number $Y_k(z)$ of descendants she eventually has in generation $\ee^{\lambda (k+1)}$.

We get a lower bound for $Y_k(z)$ by considering, instead of $z$, a virtual individual $\widetilde{z}$ in the same generation $\ee^{\lambda k}$ but positionned on the barrier at $\widetilde{V}(z):=a \ee^{\lambda k/3}\ge V(z)$. The number and displacements of the descendants of $\widetilde{z}$ are exactly the same as those of $z$. Obviously, for any $u>z$, $V(\widetilde{u})=V(u)+a \ee^{\lambda k/3}-V(z)\ge V(u)$. Hence the descendants of $\widetilde{z}$ are more likely to cross the barrier and be killed, which means that $Y_k(\widetilde{z})\le Y_k(z)$.

In order to apply Mogul'skii's estimate, we add a second absorbing barrier $i\mapsto (a-b) i^{1/3}$ for some $b>0$ and kill any descendant of $\widetilde{z}$ that is born below it. This way, we obtain that, almost surely, $Z_k\le Y_k(\widetilde{z})\le Y_k(z)$, where
\begin{equation*}
Z_k:=\#\left\{ u \in \cT_{\ee^{\lambda n}}:u>z, \forall \ee^{\lambda k} < i\le \ee^{\lambda (k+1)}, (a-b)i^{1/3}\le V(\widetilde{u_i})\le a i^{1/3}\right\}.
\end{equation*}
It is clear that $Z_k$ depends on $z$ but its law and in particular $A_k:=\p (Z_k\ge v_k)\le \p(Y_k(z)\ge v_k)$ do not.

We define, for any $n\ge1$:
\begin{equation*}
\mathcal{P}_{n}:= \p \left(\forall 1\le k\le n, \# \!\left\{ u \in \cT_\ee^{\lambda k}: \forall i\le \ee^{\lambda k}, V(u_i)\le a i^{1/3}\right\}\ge v_{k-1} \right).
\end{equation*}
If $1\le n_0\le n$, then we have:
\begin{equation*}
\mathcal{P}_{n+1}\ge\mathcal{P}_n \left(1-(1-A_n)^{v_{n-1}}\right).
\end{equation*}
By induction, we obtain:
\begin{equation*}
\mathcal{P}_{n}\ge\mathcal{P}_{n_0} \prod_{k=n_0}^{n-1}\left(1-(1-A_k)^{v_{k-1}}\right)\ge\mathcal{P}_{n_0} \prod_{k=n_0}^n\left(1-\ee^{-v_{k-1} A_k}\right).
\end{equation*}

This makes Proposition \ref{survivors_number} a consequence of the following lemma:
\begin{lemma}
If $a>a_c$, and $\lambda$ is large enough and such that $\ee^{\lambda}\in \bN$, then
\begin{equation}\label{cond_survie}
\sum_{k=0}^\infty \ee^{-v_k A_{k+1}}<+\infty.
\end{equation}
\end{lemma}

Fix $\theta\in(0,1)$, for example $\theta=\frac{1}{2}$. The Paley-Zygmund inequality, with $v_k:=\theta \e[Z_k]$ gives
\begin{equation}\label{pz}
A_k \ge (1-\theta)^2\frac{\left(\e[Z_k]\right)^2}{\e[Z_k^2]}.
\end{equation}

\subsection{Estimate of $\e[Z_k]$}

 We set $k \ge 0$ and consider only the descendance of an initial ancestor $u^0$ starting at time $\ee^{\lambda k}$ at position $a \ee^{\lambda k/3}$ over $\ell_k:=\ee ^{\lambda (k+1)}-\ee ^{\lambda k}$ generations. The individuals of generation $i$ are killed and have no descendance if they are out of the interval:
\begin{equation*}
 I_i:= [(a-b) i^{1/3}, a i^{1/3}].
\end{equation*}

With $(S_i)_{0\le i\le \ell_k}$ the zero-mean random walk starting from $0$ defined in Lemma \ref{afo},
\begin{eqnarray}
\e [Z_k] & = & \e \left[\sum_{u>u^0,|u|=\ee^{\lambda (k+1)}} \ind_{\left\{\forall \ee^{\lambda k} < i\le \ee^{\lambda (k+1)}, V(u_i)\in I_i \right\}}\right]\nonumber\\
& = & \e \left[\ee^{S_{\ell_k}} \ind_{\left\{\forall i\le \ell_k, S_i\in I_{\ee^{\lambda k}+i}\right\}}\right]\label{afo_ub}\\
& \ge & \ee^{a \ee ^{\lambda (k+1)/3} -b_{\ee ^{\lambda (k+1)}}} \p \left(\forall i\le \ell_k, S_i\in I_{\ee^{\lambda k}+i}\right).\label{lb_e_z_k_un}
\end{eqnarray}

Define for some continuous function $g:[0,1] \mapsto (0,+\infty)$ and $0\le i \le \ell_k$:
\begin{equation}\label{how_to_apply_mse}
g_2(t):=a\left(t+\frac{\ee^{\lambda k}}{\ell_k}\right)^{1/3},g(t):=b\left(t+\frac{\ee^{\lambda k}}{\ell_k}\right)^{1/3}, g_1(t):=g_2(t)-g(t).
\end{equation}
Then, the lower bound in Corrolary \ref{mse_local} implies that $\forall \varepsilon>0$, $\exists k_0\ge 1$, $\forall k\ge k_0$,
\begin{equation}
\e [Z_k] \ge \exp \left(\ell_k^{1/3}\left[g_2(1)- \frac{\pi^2\sigma^2}{2}\int_0^1 \frac{\d t}{g(t)^2}-\varepsilon\right]\right).\label{e_z_k}
\end{equation}

\subsection{Estimate of $\e[Z_k^2]$}
We split the double sum over $u,v\in \cT$ according to the generation $j$ of $u_j=u\wedge v \in \cT$ the lowest common ancestor of $u$ and $v$:
\begin{equation}\label{mom_deux_sum_B_k,j}
\e [Z_k^2] = \e \left[\sum_{\stackrel{u>z,v>z}{|u|=|v|=\ee ^{\lambda (k+1)}}}\ind_{\left\{\forall \ee^{\lambda k} < i\le \ee^{\lambda (k+1)}, V(u_i)\in I_i, V(v_i)\in I_i\right\}} \right]=  \sum_{j=0}^{\ell_k} B_{k,j},
 \end{equation}
where
\begin{equation}\label{B_k,j_egal}
B_{k,j}:=\e \left[\sum _{u>z,|u|=\ee ^{\lambda (k+1)}} \ind_{\left\{\forall \ee^{\lambda k} < i\le \ee^{\lambda (k+1)}, V(u_i)\in I_i\right\}} \sum _{v>u_j,|v|=\ee ^{\lambda (k+1)}} \ind_{\left\{\forall \ee^{\lambda k}+j < i\le \ee^{\lambda (k+1)}, V(v_i)\in I_i\right\}}\right].
\end{equation}
Thanks to Lemma \ref{afo}, we have:
\begin{eqnarray}
h_{k,j}(x)&:=&\e \left[\sum _{v>u_j,|v|=\ee ^{\lambda (k+1)}} \ind_{\left\{\forall \ee^{\lambda k}+j < i\le \ee^{\lambda (k+1)}, V(v_i)\in I_i\right\}} \middle|V(z)=x\right]\nonumber\\
& = & \e \left[\ee^{S_{\ell_k-j}} \ind_{\left\{\forall 0 < i\le \ell_k - j, x+S_i\in I_{\ee ^{\lambda k}+i}\right\}}\right]\nonumber\\
& \le & \ee^{a \left(\ee^{\lambda (k+1)/3}-(\ee^{\lambda k}+j)^{1/3}\right)+b_{\ee^{\lambda k}+j}} \p\left(\forall 0 < i\le \ell_k - j, x+S_i\in I_{\ee ^{\lambda k+1}+i}\right).\label{h_kj}
\end{eqnarray}
Note that by equation (\ref{B_k,j_egal}),
\begin{equation*}
B_{k,j} \le \sup_{x\in I_{\ee ^{\lambda k}+j}} h_{k,j}(x) \e \left[ Z_k\right].
\end{equation*}
We proceed to estimate
\begin{equation*}
\sup_{x\in I_{\ee ^{\lambda k}+j}} \p\left(\forall 0 < i\le \ell_k - j, x+S_i\in I_{\ee ^{\lambda k+1}+i}\right).
\end{equation*}
Because Theorem \ref{mse} can only be applied with a specified initial position, we will not apply it directly with the functions specified in equations (\ref{how_to_apply_mse}) but we keep $b_{\ee ^{\lambda k}+i}:=\ell_k^{1/3}g(\frac{i}{\ell_k})$ for the same continuous $g: [0,1]\mapsto (0,+\infty)$ as previously in this section and $g_2(t):=a\left(t+\frac{\ee^{\lambda k}}{\ell_k}\right)^{1/3}=a\left(t+\frac{1}{\ee^\lambda -1}\right)^{1/3}$. The upper bound (\ref{h_kj}) gives:
\begin{equation}
h_{k,j}(x)\le \ee^{\ell_k^{1/3}\left(g_2(1) - g_2(\frac{j}{\ell_k}) + g(\frac{j}{\ell_k}) \right)} \p\left(\forall 0 < i\le \ell_k - j, x+S_i\in I_{\ee ^{\lambda k+1}+i}\right).\label{h_kj_g}
\end{equation}

For any $\varepsilon>0$ we set $M\in\bN$. We then split the interval $I_{\ee ^{\lambda k}+j}=\ell_k^{1/3} [g_2(\frac{j}{\ell_k})-g(\frac{j}{l_k}),g_2(\frac{j}{\ell_k})]$ into $M$ intervals $J_{k,j}^m:=\ell_k^{1/3} [g_2(\frac{j}{\ell_k})-\frac{m+1}{M}g(\frac{j}{\ell_k}),g_2(\frac{j}{\ell_k})-\frac{m}{M}g(\frac{j}{\ell_k})], 0\le m < M$.

For $y\in \frac{1}{\ell_k^{1/3}}J_{k,j}^m=[g_2(\frac{j}{\ell_k})-\frac{m+1}{M}g(\frac{j}{\ell_k}), g_2(\frac{j}{\ell_k})-\frac{m}{M}g(\frac{j}{\ell_k})]$,
\begin{eqnarray}
&&\p\left(\forall 0 < i\le \ell_k - j,\; \ell_k^{1/3}y +S_i\in I_{\ee ^{\lambda k}+i}\right)\nonumber\\
&=&\p\left(\forall 0 < i\le \ell_k - j,\; g_2(\frac{j}{\ell_k})-g(\frac{j}{\ell_k})\le y+\frac{S_i}{\ell_k}\le g_2(\frac{j}{\ell_k})\right)\nonumber\\
&\le & \p\left(\forall 0 < i\le \ell_k - j,\; \frac{m+1-M}{M}g(\frac{j}{\ell_k})\le\frac{S_i}{\ell_k}\le \frac{m}{M}g(\frac{j}{\ell_k})\right).\label{ub_p_Z^deux}
\end{eqnarray}

Now let $j=:\ell_k \alpha$, for a fixed $\alpha \in (0,1)$ such that $\ee^{\lambda k^0}\alpha \in \bN$ for some $k^0\in \bN$, and apply Theorem \ref{mse} using a trick similar to the one we used when we estimated $R_j(n)$ in the preceding section, by replacing any boundary function $f$ by the corresponding $f^\alpha$ defined by $f^\alpha(t):=(1-\alpha)^{-1/3}g(\alpha + (1-\alpha) t)$. We use in particular the equality
\begin{equation*}
(1-\alpha)^{1/3} \int_0^1 \frac{\d t}{g^\alpha(t)^2}=\int_\alpha^1 \frac{\d t}{g(t)^2}.
\end{equation*}
This way, we obtain that, for any $0\le m \le M$, as $k\rightarrow \infty$:
\begin{equation*}
\frac{1}{\ell_k^{1/3}}\log \p\left(\forall 0 < i\le \ell_k - j, \frac{m+1-M}{M}g(\frac{j}{\ell_k})\le\frac{S_i}{\ell_k}\le \frac{m}{M}g(\frac{j}{\ell_k})\right)\rightarrow -\frac{\pi^2 \sigma^2}{2}\frac{M^2}{(M-1)^2} \int_\alpha^1 \frac{\d t}{g(t)^2}.
\end{equation*}
Then for any $0\le m \le M$, we have for $k$ larger than some $k_\alpha \in \bN$:
\begin{equation*}
\frac{1}{\ell_k^{1/3}}\log \p\left(\forall 0 < i\le \ell_k - j, \frac{m+1-M}{M}g(\frac{j}{\ell_k})\le\frac{S_i}{\ell_k}\le \frac{m}{M}g(\frac{j}{\ell_k})\right)\le -\frac{\pi^2 \sigma^2}{2}\frac{M^2}{(M-1)^2} \int_\alpha^1 \frac{\d t}{g(t)^2}+\varepsilon.
\end{equation*}
Now we set the value $M$. We choose it large enough to have:
\begin{equation*}
\frac{\pi^2 \sigma^2}{2}\left(\frac{M^2}{(M-1)^2}-1\right)\int_\alpha^1 \frac{\d t}{g(t)^2}\le \varepsilon.
\end{equation*}
Plugging the last two inequalities into (\ref{ub_p_Z^deux}), we obtain for any $0\le m \le M$,
\begin{equation*}
\frac{1}{\ell_k^{1/3}}\log \sup_{x\in J_{k,j}^m} \p\left(\forall 0 < i\le \ell_k - j, x+S_i\in I_{\ee ^{\lambda k+1}+i}\right) \le -\frac{\pi^2 \sigma^2}{2}  \int_\alpha^1 \frac{\d t}{g(t)^2} + 2\varepsilon.
\end{equation*}
Since $I_{\ee ^{\lambda k}+j}= \bigcup_{0\le m <M}J_{k,j}^m$, this implies:
\begin{equation*}
\frac{1}{\ell_k^{1/3}}\log \sup_{x\in I_{\ee ^{\lambda k}+j}} \p\left(\forall 0 < i\le \ell_k - j, x+S_i\in I_{\ee ^{\lambda k+1}+i}\right) \le -\frac{\pi^2 \sigma^2}{2}  \int_\alpha^1 \frac{\d t}{g(t)^2} + 2\varepsilon.
\end{equation*}
Combining this with equation (\ref{e_z_k}) and inequality (\ref{h_kj_g}), we obtain for $k\ge \max(k_0,k_\alpha)$:
\begin{equation*}
\frac{1}{\ell_k^{1/3}}\log \frac{B_{k,j}}{\left(\e[Z_k]\right)^2}\le a \left[-(\ee^{\lambda k}+j)^{1/3}+\ee^{\lambda k/3}\right]+\ell_k^{1/3} \left[g(\alpha)+\frac{\pi^2 \sigma^2}{2}\int_0^\alpha \frac{\d t}{g(t)^2}+3\varepsilon\right].
\end{equation*}
The constant $k_\alpha$ depends on $\alpha$, but with an argument similar to the one we used in the previous section, (notice that in the case $\alpha =1$ the last inequality still holds, take some $N=\ee^{\lambda k^1}$ for some $k^1\in \bN$ and apply Theorem \ref{mse} $N-1$ times for $\alpha=\frac{n}{N}$, $0<n<N$), we obtain that for $k$ large enough, for any $0<\alpha=\frac{j}{\ell_k}\le 1$:
\begin{equation*}
\log \frac{B_{k,\alpha \ell_k}}{\left(\e[Z_k]\right)^2}\le \ell_k^{1/3} \left[g_2(0) -g_2(\alpha) + g(\alpha)+\frac{\pi^2 \sigma^2}{2}\int_0^\alpha \frac{\d t}{g(t)^2}+3\varepsilon\right].
\end{equation*}
With $\ell_k=\ee^{\lambda} \ell_{k-1}$, we get that for $k$ large enough, for any $0<\alpha=\frac{j}{\ell_k}\le 1$:
\begin{eqnarray*}
\frac{1}{\ell_k^{1/3}}\log \frac{B_{k,\alpha \ell_k}}{\left(\e[Z_k]\right)^2 v_k}&\le& \ell_k^{1/3} \left[g_2(0) -g_2(\alpha) + g(\alpha)+\frac{\pi^2 \sigma^2}{2}\int_0^\alpha \frac{\d t}{g(t)^2}+4\varepsilon\right]\\
&&+\ell_{k-1}^{1/3} \left[g_2(0)-g_2(1) +\int_0^1 \frac{\d t}{g(t)^2}+\varepsilon\right]. \nonumber
\end{eqnarray*}

We combine this with equations (\ref{cond_survie}), (\ref{pz}) and (\ref{mom_deux_sum_B_k,j}) and choose $\varepsilon$ small enough to obtain:
\begin{equation}\label{impl_G_lambda}
\max_{0\le \alpha \le 1} G_\lambda(\alpha)<0 \Rightarrow \p\left(\exists u \in \lod, \forall i\ge 1, V(u_i)\le a i^{1/3} \right)>0
\end{equation}
where
\begin{equation*}
G_\lambda(\alpha):= g_2(0) -g_2(\alpha) + g(\alpha)+\frac{\pi^2 \sigma^2}{2}\int_0^\alpha \frac{\d t}{g(t)^2}+\ee^{-\lambda/3} \left[g_2(0)-g_2(1) +\frac{\pi^2 \sigma^2}{2} \int_0^1 \frac{\d t}{g(t)^2}\right].
\end{equation*}

\subsection{Choice of $g$ and $\lambda$ for the lower bound}

We denote
\begin{equation*}
\forall t\in [0,1], f(t):=\left(t+\frac{1}{\ee^\lambda -1}\right)^{1/3}.
\end{equation*}
We have $g_2=a f$. We choose for the width of the pipe the function $g:=b f$.
Then
\begin{equation*}
G_\lambda(\alpha)= a f(0) + (b-a) f(\alpha) +\frac{\pi^2 \sigma^2}{2 b^2}\int_0^\alpha \frac{\d t}{f(t)^2}+\ee^{-\lambda/3} \left[a f(0)-a f(1) +\frac{\pi^2 \sigma^2}{2 b^2} \int_0^1 \frac{\d t}{ f(t)^2}\right].
\end{equation*}
Since $f(1)=\ee^{\lambda/3}f(0)$ and $f'=\frac{1}{3}f^{-2}$, this becomes:
\begin{equation*}
G_\lambda(\alpha) = \left(b+\frac{3\pi^2 \sigma^2}{2 b^2}-a\right) f(\alpha)+\ee^{-\lambda/3} \left[a f(0)-\frac{3\pi^2 \sigma^2}{2 b^2} f(0)\right].
\end{equation*}
Assuming $a>a_c$, we can choose $b$ such that $b+\frac{3\pi^2 \sigma^2}{2 b^2}<a$. Since $f$ is increasing on $[0,1]$,
\begin{equation*}
\max_{0\le \alpha \le 1} G_\lambda(\alpha) = G_\lambda(0)=f(0) \left[\left(b+\frac{3\pi^2 \sigma^2}{2 b^2}-a\right) +\ee^{-\lambda/3} \left(a -\frac{3\pi^2 \sigma^2}{2 b^2} \right)\right].
\end{equation*}
This value is negative for sufficiently large $\lambda$ (that we can choose such that we also have $\ee^{\lambda}\in \bN$), which, in view of (\ref{impl_G_lambda}), completes the proof.

\section{The extinction rate}\label{section_extinction_rate}

Throughout this section, we assume $a<a_c$.

\subsection{Removing the condition $n\in N\bN$}\label{suppress_restriction}
In Section \ref{section_upper_bound}, we assumed that $n$ was always a multiple of $N$.
There were two reasons for taking such a restriction over the values of $N$ :
The first one is that it avoided heavier notation with integer parts. The second one is that it allowed to obtain equation (\ref{r_j_deux}) faster.

This restriction does not matter if we only want to prove Theorem \ref{result}, but in order to prove rigorously Proposition \ref{extinction_rate}, we have to remove it.
This can be achieved at the cost of some extra $\varepsilon$ in equation (\ref{r_j_deux}).

Let $\alpha\in (0,1]$, $\varepsilon>0$ and, for any $n\in\bN$, $j=j(n):=\lfloor \alpha n \rfloor +1$.
Let $g$ and $\widetilde{g}$ be some continuous functions $[0,1]\rightarrow [0,+\infty)$.
We define, for any $t\in[0,1]$, $\widetilde{g}_\alpha(t):=\alpha^{-1/3}\widetilde{g}(\alpha t)$.
We can now apply Corollary \ref{mse_ub} so that for $n$ large enough, we have
\begin{equation*}
\frac{1}{(j-1)^{1/3}} \log \p  \left(\forall i<j, a i^{1/3}-(j-1)^{1/3}\widetilde{g}_\alpha(\frac{i}{j-1})\le S_i\le a i^{1/3} \right) \le -\frac{\pi^2 \sigma^2}{2}\int_0^1\frac{\d t}{\widetilde{g}_\alpha (t)^2}+\varepsilon. \nonumber
\end{equation*}
Using the definition of $\widetilde{g}_\alpha$, the left-hand side of this inequality becomes:
\begin{equation*}
\widetilde{\lhs}:=\frac{1}{(j-1)^{1/3}} \log \p  \left(\forall i<j, a i^{1/3}-\left(\frac{j-1}{\alpha}\right)^{1/3}\widetilde{g}(\frac{\alpha i}{j-1})\le S_i\le a i^{1/3} \right);
\end{equation*}
and the right-hand side:
\begin{equation*}
\widetilde{\rhs}:=-\frac{\pi^2 \sigma^2}{2\alpha^{1/3}}\int_0^\alpha\frac{\d u}{\widetilde{g}(u)^2}+\varepsilon.
\end{equation*}
What we want is to obtain an inequality of the form $\lhs\le \rhs+\varepsilon$, where
\begin{equation*}
\lhs:=\frac{1}{(j-1)^{1/3}} \log \p  \left(\forall i<j, a i^{1/3}-n^{1/3}g(\frac{i}{n})\le S_i\le a i^{1/3} \right);
\end{equation*}
and\begin{equation*}
\rhs:=-\frac{\pi^2 \sigma^2}{2\alpha^{1/3}}\int_0^\alpha\frac{\d u}{\widetilde{g}(u)^2}+\varepsilon.
\end{equation*}
It is sufficient to have $\lhs\le \widetilde{\lhs}$ and $\widetilde{\rhs}\le \rhs$.
The first inequality holds as soon as we have
\begin{equation*}
\forall i<j, n^{1/3}g(\frac{i}{n}) \le \left(\frac{j-1}{\alpha}\right)^{1/3}\widetilde{g}(\frac{\alpha i}{j-1}).
\end{equation*}
Since $0\le \frac{\alpha i}{j-1}-\frac{i}{n} \le\frac{1}{\alpha n}$ and $\frac{j-1}{\alpha n}\ge 1-\frac{1}{\alpha n}$, we set:
\begin{equation*}
\forall t\in[0,1], \widetilde{g}(t):=(1-\eta)^{-1/3} \max_{\max(t-\eta,0)\le u \le t} g(u),
\end{equation*}
where $\eta>0$ is arbitrary small.
This works when $n\ge \frac{1}{\eta \alpha}$.
Since this is to be applied with $\alpha\in\{\frac{1}{N}, \frac{2}{N}, \dots, \frac{N-1}{N},1\}$ for some (large) fixed $N$, we have just proved that $\lhs\le \widetilde{\lhs}$ holds when $n$ is large enough.

According to our choice of $\widetilde{g}$, $\widetilde{\rhs}\le \rhs$ holds for $\eta$ close enough to zero.

Consequently inequality (\ref{r_j_deux}) (with one extra $\varepsilon$) still holds without the condition $n\in N\bN$, and it is easy to verify that all the arguments and results of Section \ref{section_upper_bound} are still valid.

\subsection{Upper bound}
It follows from the computations of Section \ref{section_upper_bound} that, for any continuous function $g :[0,1]\mapsto [0,+\infty)$ such that $g(0)=1$,
\begin{equation*}
\limsup_n \frac{1}{n^{1/3}}\log \p\left(\exists u \in \cT_n, \forall i\le n, V(u_i)\le a i^{1/3} \right)\le -c_g,
\end{equation*}
where
\begin{equation*}
c_g:=\min_{0\le t\le1}\left( g(t)+\frac{\pi^2 \sigma^2}{2}\int_0^t\frac{\d u}{g(u)^2}-a t^{1/3}\right).
\end{equation*}

The best choice for $g$ is the one described in the end of Section \ref{section_upper_bound}: it is the solution of the integral equation (\ref{equ_int}) with $\st=c_g$ such that $g(1)=0$ (or equivalently, $t_{max}=1$). We can make this choice thanks to Proposition \ref{inv_prop} and Proposition \ref{study_equ_diff}(iii).

\subsection{Lower bound}

We directly apply the Paley-Zygmund inequality to the number $W_n$ of individuals $u\in\cT_n$ such that.
\begin{equation*}\forall i\le n, a i^{1/3}-n^{1/3}g(\frac{i}{n})\ge V(u) \le a i^{1/3}.\end{equation*}

Following the computations of Section $\ref{section_survivors}$, we obtain

\begin{equation*}
\liminf_n \frac{1}{n^{1/3}}\log \p\left(\exists u \in \cT_n, \forall i\le n, V(u_i)\le a i^{1/3} \right)\ge\liminf_n \frac{1}{n^{1/3}}\log\left( \p (W_n\ge 1)\ge\right)\ge -d_g,
\end{equation*}
where
\begin{equation*}
d_g:=\max_{0\le t\le1}\left( g(t)+\frac{\pi^2 \sigma^2}{2}\int_0^t\frac{\d u}{g(u)^2}-a t^{1/3}\right).
\end{equation*}

The optimal $g$ would be exactly the same as in the upper bound, except that we are forced to take approximations because $g$ must be strictly positive on $[0,1]$.
Since this optimal $g$ is such that $g(t)+\frac{\pi^2 \sigma^2}{2}\int_0^t\frac{\d u}{g(u)^2}-a t^{1/3}$ does not depend on $t$, we have proved
\begin{equation*}
c:=\sup_g c_g=\inf_g d_g.
\end{equation*}
This completes the proof of Proposition \ref{extinction_rate}.

\appendix
\section{Extension of the results to the non critical case}\label{app_reduc}

\subsection{When the reduction to the critical case is possible}
We assume $\Phi(0)<+\infty$ and $\zeta:=\sup\{t:\Phi(t)<+\infty\}>0$.
We define, for $0<t<\zeta$,
\begin{equation*}F(t):=\frac{\Psi(t)}{t}\end{equation*}
\begin{lemma}\label{F}
If equation (\ref{hyp_red}) does not hold for any $t^*$, then $\forall 0<t<\zeta, t\Psi'(t)<\Psi(t)$ and $F$ is decreasing and convex.
\end{lemma}
\begin{proof}
 For $t>0$ very small, $t\Psi'(t)-\Psi(t)<0$, since by convexity of $\Psi$, either $\Psi'(0)\in \r$ or $\Psi'(0)=-\infty$.
$\forall t>0$, $t\Psi'(t)-\Psi(t)\neq0$ by hypothesis, and so by continuity, it is always negative.
As a consequence, for any $0<t<\zeta$,
\begin{eqnarray*}F'(t)&=&\frac{t\Psi'(t)-\Psi(t)}{t^2}<0;\\
F''(t)&=&\frac{t^2\Psi''(t)+2(\Psi(t)-t\Psi'(t))}{t^3}>0;\end{eqnarray*}
which ends the proof of the lemma.
\end{proof}
We are now ready to determine whether a branching random walk can be applied the reduction of Section \ref{subsec_red}.
It easy to construct examples for any of the cases studied below.

\subsubsection{Case $\zeta<+\infty$}

a) If $\zeta<+\infty$ and $\Phi(\zeta)=+\infty$, then, by Fatou's lemma, $\lim_{t\rightarrow \zeta}\Phi(t)=+\infty$.

b) If $\zeta<+\infty$, $\Phi(\zeta)<+\infty$ and $\Phi'(\zeta)=+\infty$.

In these two cases, it is easy to deduce from the lemma that we can find some $t^*\in (0,\zeta)$ such that equation (\ref{hyp_red}) holds.

When $\zeta<+\infty$, $\Phi(\zeta)<+\infty$ and $\Phi'(\zeta)<+\infty$, it depends on the sign of $\zeta\Psi'(\zeta)-\Psi(\zeta)$:

c) If $\zeta\Psi'(\zeta)-\Psi(\zeta)>0$, then by continuity we can find some $t^*\in (0,\zeta)$ such that equation (\ref{hyp_red}) holds.

d) If $\zeta\Psi'(\zeta)-\Psi(\zeta)<0$, then it is impossible to find such a $t^*$ (because $t\mapsto t\Psi'(t)-\Psi(t)$ is increasing) and consequently the reduction to the critical case does not work.

e) If $\zeta\Psi'(\zeta)-\Psi(\zeta)=0$, then $t^*=\zeta$ works, but we still have to check whether $\widetilde{\sigma}^2$ is finite or not.

It is easy to construct examples fitting any of these five cases.

\subsubsection{Case $\zeta=+\infty$}

In this case, we can be much more precise and tell whether we can find some $t^*\in (0,\zeta)$ such that equation (\ref{hyp_red}) holds directly from the intensity measure $\mu$ of the point process $\{\xi_u,|u|=1\}$

Define $x_{min}=\inf\{x\in \r, \mu( (-\infty,x))>0\}$ the minimum of the support of $\mu$ (and $-\infty$ if $\mu$ is not lower bounded).
It is clear that $\lim_{t\rightarrow +\infty}F(t)=-x_{min}$.
If $x_{min}>-\infty$ we will consider $\mu(\{x_{min}\})$ the mass of the eventual atom of $\mu$ in $x_{min}$.
We can now state :
\begin{proposition}
There is some $t^*\in (0,\zeta)$ such that equation (\ref{hyp_red}) holds if and only if $x_{min}>-\infty$ or $\mu(\{x_{min}\})<1$.
\end{proposition}

\begin{proof}
We distinguish four cases:

a) If $x_{min}=-\infty$, then $\lim_{+\infty} F=-x_{min}=+\infty$. Consequently $F$ can not be decreasing.

b) If $x_{min}>-\infty$ and $\mu(\{x_{min}\})=0$. We still have $\lim_{+\infty} F=-x_{min}$.
Almost surely, for all $u\in \cT_1$, $\xi_u<x_{min}$, hence $\sum_{|u|=1}\ee^{t(x_{min}-\xi_u)}\rightarrow 0$ as $t\rightarrow +\infty$. By the monotone convergence theorem $\lim_{t\rightarrow +\infty}\e\left[\sum_{|u|=1}\ee^{t(x_{min}-\xi_u)}\right]= 0$.
This implies that for $t$ large enough,
\begin{equation*}
F(t)=-x_{min}+\frac{1}{t}\log \left(\e\left[\sum_{|u|=1}\ee^{t(x_{min}-\xi_u)}\right]\right)<-x_{min}.
\end{equation*}

c) If $x_{min}>-\infty$ and $0<\mu(\{x_{min}\})<1$, we can write
\begin{equation*}
F(t)=-x_{min}+\frac{1}{t}\log \left(\mu(\{x_{min}\})+\varepsilon(t)\right)
\end{equation*}
where $\varepsilon(t):=\e\left[\sum_{|u|=1}\ind_{\xi_u>x_{min}} \ee^{t(x_{min}-\xi_u)}\right]$ decreases to $0$ as $t$ increases to infinity.
Like in case b), for $t$ large enough, $\log \left(\mu(\{x_{min}\})+\varepsilon(t)\right)<0$ and $F(t)<-x_{min}$.

In these three cases, we have proved that $F$ is not decreasing, and we conclude thanks to Lemma \ref{F}: there is some $t^*\in (0,\zeta)$ such that equation (\ref{hyp_red}) hold, since otherwise the lemma would imply that $F$ is decreasing, which is false.

d) When $x_{min}>-\infty$ and $\mu(\{x_{min}\})\ge 1$, we still have
\begin{equation*}
\Psi(t)=- t x_{min}+\log \left(\mu(\{x_{min}\})+\varepsilon(t)\right)\ge-x_{min},
\end{equation*}
where $\varepsilon(t):=\e\left[\sum_{|u|=1}\ind_{\xi_u>x_{min}} \ee^{t(x_{min}-\xi_u)}\right]$ decreases to $0$ as $t$ increases to infinity.
By convexity, $\Psi'$ increases, and so converges to $-x_{min}$ as $t$ goes to $+\infty$.

Finally, for any $t>0$, we proved that  $\Psi'(t)<-x_{min}$ whereas $\Psi(t)\ge-t x_{min}$, hence $t\Psi'(t)<\Psi(t)$.
\end{proof}

\nocite{*}
\bibliographystyle{plain}
\bibliography{biblioBRW}

\small{Bruno Jaffuel

LPMA

Universit\'e Paris VI

4 Place Jussieu

F-75252 Paris Cedex 05

France

bruno.jaffuel@upmc.fr}

\end{document}